\documentclass[11pt,leqno]{amsart}

\usepackage{amsthm}
\usepackage{amscd}
\usepackage{amsfonts}
\usepackage{amssymb}
\usepackage{amsgen}
\usepackage{amsmath}
\usepackage{amsopn}
\usepackage{verbatim}
\usepackage{xypic}
\usepackage{xspace}
\usepackage{multicol}
\usepackage{url}
\usepackage{upref}

\theoremstyle{plain}
\newtheorem{thm}{Theorem}[section]
\newtheorem{lem}[thm]{Lemma}

\theoremstyle{definition}

\theoremstyle{remark}

 \font\cyr=wncyr10
 \newcommand{\nc}{\newcommand}

\nc{\per}[1]{\underset{#1}{\boldsymbol \pi}\,}

 \nc{\gf}{{\varphi}}
 \nc{\MT}{{\rm MT}}
  \nc{\bgz}{{\bar{\gz}}}
 \nc{\wt}{{\rm wt}}
 \nc{\wht}{{\widehat}}
 \nc{\bwg}{{\bigwedge}}
 \nc{\mmu}{{\boldsymbol{\mu}}}
 \nc{\mal}{{{\scriptstyle \maltese}}}
 \nc{\fA}{{\mathfrak A}}
 \nc{\HH}{{\mathfrak H}}
 \nc{\ra}{\rightarrow}
 \nc{\ors}{{\vec s\,}}
 \nc{\os}{{\overset}}
 \nc{\G}{{\mathbb G}}
 \nc{\Z}{{\mathbb Z}}
 \nc{\R}{{\mathbb R}}
 \nc{\N}{{\mathbb N}}
 \nc{\ZN}{{\mathbb Z_{\ge 0}}}
 \nc{\Q}{{\mathbb Q}}
 \nc{\C}{{\mathbb C}}
 \nc{\Cnn}{{\mathbb C}_{\ge 0}}
 \nc{\Cp}{{\mathbb C}_{>0}}
 
 \nc{\tB}{{\tilde B}}
 \nc{\tI}{{\tilde I}}
 \nc{\tJ}{{\tilde J}}
 \nc{\tK}{{\tilde K}}
 \nc{\Li}{{\rm Li}}
 \nc{\suf}{{\ast\,}}
 \nc{\sufq}{{\ast_q\,}}
 \nc{\gam}{{\gamma}}
 \nc{\gG}{{\Gamma}}
 \nc{\om}{{\omega}}
 \nc{\vep}{{\varepsilon}}
 \nc{\ga}{{\alpha}}
 \nc{\gl}{{\lambda}}
 \nc{\gb}{{\beta}}
 \nc{\gd}{{\delta}}
 \nc{\gs}{{\sigma}}
 \nc{\gS}{{\Sigma}}

 \nc{\gk}{{\kappa}}
 \nc{\tgz}{{\tilde{\zeta}}}
 \nc{\gO}{{\Omega}}
 \nc{\sif}{{\mathcal S}}
 \nc{\gt}{{\tau}}
\nc{\gz}{{\zeta}}
 \nc{\Lra}{\Longrightarrow}
 \nc{\lra}{\longrightarrow}
 \nc{\fS}{{\mathfrak S}}
 \nc{\DD}{{\mathfrak D}}

 \nc{\Llra}{\Longleftrightarrow}
 \nc{\ol}{\overline}
 \nc{\lms}{\longmapsto}
 \nc{\cv}{{{\mathsf c}{\mathsf v}}}
 \nc{\zq}{{\zeta_q}}
 \nc\qup{{q\uparrow 1}}
 \nc{\us}{\underset}
 \nc{\tn}{{\tilde{n}}}
 \nc{\gD}{{\Delta}}
 \nc{\bi}{{\bf i}}
 \nc{\bfone}{{\bf 1}}
 \nc{\bfa}{{\bf a}}
 \nc{\bfb}{{\bf b}}
 \nc{\bfc}{{\bf c}}
 \nc{\bfd}{{\bf d}}
 \nc{\bfe}{{\bf e}}
 \nc{\bff}{{\bf f}}
 \nc{\bfg}{{\bf g}}
 \nc{\bfh}{{\bf h}}
 \nc{\bfi}{{\bf i}}
 \nc{\bfj}{{\bf j}}
 \nc{\bfn}{{\bf n}}
 \nc{\bfl}{{\bf l}}
 \nc{\bfk}{{\bf k}}
 \nc{\bfm}{{\bf m}}
 \nc{\bfo}{{\bf o}}
 \nc{\bfp}{{\bf p}}
 \nc{\bfq}{{\bf q}}
 \nc{\bfr}{{\bf r}}
 \nc{\tbfs}{{\tilde{\bf s}}}
 \nc{\bfs}{{\bf s}}
 \nc{\hbfs}{{\hat{\bf s}}}
 \nc{\hs}{{\hat{s}}}
 \nc{\ts}{\tilde{s}}
 \nc{\bft}{{\bf t}}
 \nc{\bfu}{{\bf u}}
 \nc{\bfv}{{\bf v}}
 \nc{\bfw}{{\bf w}}
 \nc{\bfx}{{\bf x}}
 \nc{\bfy}{{\bf y}}
 \nc{\bfz}{{\bf z}}
 \nc{\bfB}{{\bf B}}
 \nc{\bfP}{{\bf P}}
 \nc{\bfQ}{{\bf Q}}
 \nc{\bfY}{{\bf Y}}
 \nc{\bfgb}{{\boldsymbol \gb}}
 \nc{\bfga}{{\boldsymbol \ga}}
 \nc{\bfrho}{{\boldsymbol \rho}}
 \nc{\bfchi}{{\boldsymbol \chi}}
 \nc{\QX}{{\Q\langle \bfX\rangle}}
 \nc{\QY}{{\Q\langle \bfY\rangle}}
 \nc{\CX}{{\C\langle \bfX\rangle}}
 \nc{\CY}{{\C\langle \bfY\rangle}}
 \nc{\QXX}{{\Q\langle\!\langle \bfX\rangle\!\rangle}}
 \nc{\QYY}{{\Q\langle\!\langle \bfY\rangle\!\rangle}}
 \nc{\CXX}{{\C\langle\!\langle \bfX\rangle\!\rangle}}
 \nc{\CYY}{{\C\langle\!\langle \bfY\rangle\!\rangle}}

 \nc{\bbA}{{\mathbb A}}
 \nc{\bbB}{{\mathbb B}}
 \nc{\bbC}{{\mathbb C}}
 \nc{\bbD}{{\mathbb D}}
 \nc{\bbE}{{\mathbb E}}
 \nc{\bbF}{{\mathbb F}}
 \nc{\bbG}{{\mathbb G}}
 \nc{\bbH}{{\mathbb H}}
 \nc{\bbI}{{\mathbb I}}
 \nc{\bbJ}{{\mathbb J}}
 \nc{\bbK}{{\mathbb K}}
 \nc{\bbL}{{\mathbb L}}
 \nc{\bbM}{{\mathbb M}}
 \nc{\bbN}{{\mathbb N}}
 \nc{\bbO}{{\mathbb O}}
 \nc{\bbP}{{\mathbb P}}
 \nc{\bbQ}{{\mathbb Q}}
 \nc{\bbR}{{\mathbb R}}
 \nc{\bbS}{{\mathbb S}}
 \nc{\bbT}{{\mathbb T}}
 \nc{\bbU}{{\mathbb U}}
 \nc{\bbV}{{\mathbb V}}
 \nc{\bbW}{{\mathbb W}}
 \nc{\bbX}{{\mathbb X}}
 \nc{\bbY}{{\mathbb Y}}
 \nc{\bbZ}{{\mathbb Z}}
 \nc{\bba}{{\mathbb a}}
 \nc{\bbb}{{\mathbb b}}
 \nc{\bbc}{{\mathbb c}}
 \nc{\bbd}{{\mathbb d}}
 \nc{\bbe}{{\mathbb e}}
 \nc{\bbf}{{\mathbb f}}
 \nc{\bbg}{{\mathbb g}}
 \nc{\bbh}{{\mathbb h}}
 \nc{\bbi}{{\mathbb i}}
 \nc{\bbk}{{\mathbb k}}
 \nc{\bbl}{{\mathbb l}}
 \nc{\bbm}{{\mathbb m}}
 \nc{\bbn}{{\mathbb n}}
 \nc{\bbo}{{\mathbb o}}
 \nc{\bbp}{{\mathbb p}}
 \nc{\bbq}{{\mathbb q}}
 \nc{\bbr}{{\mathbb r}}
 \nc{\bbs}{{\mathbb s}}
 \nc{\bbt}{{\mathbb t}}
 \nc{\bbu}{{\mathbb u}}
 \nc{\bbv}{{\mathbb v}}
 \nc{\bbw}{{\mathbb w}}
 \nc{\bbx}{{\mathbb x}}
 \nc{\bby}{{\mathbb y}}
 \nc{\bbz}{{\mathbb z}}

 \nc{\calA}{{\mathcal A}}
 \nc{\calB}{{\mathcal B}}
 \nc{\calC}{{\mathcal C}}
 \nc{\calD}{{\mathcal D}}
 \nc{\calE}{{\mathcal E}}
 \nc{\calF}{{\mathcal F}}
 \nc{\calG}{{\mathcal G}}
 \nc{\calH}{{\mathcal H}}
 \nc{\calI}{{\mathcal I}}
 \nc{\calJ}{{\mathcal J}}
 \nc{\tcalI}{{\tilde{\mathcal I}}}
 \nc{\tcalJ}{{\tilde{\mathcal J}}}
 \nc{\calK}{{\mathcal K}}
 \nc{\calL}{{\mathcal L}}
 \nc{\calM}{{\mathcal M}}
 \nc{\calN}{{\mathcal N}}
 \nc{\calO}{{\mathcal O}}
 \nc{\calP}{{\mathcal P}}
 \nc{\calQ}{{\mathcal Q}}
 \nc{\calR}{{\mathcal R}}
 \nc{\calS}{{\mathcal S}}
 \nc{\calT}{{\mathcal T}}
 \nc{\calU}{{\mathcal U}}
 \nc{\calV}{{\mathcal V}}
 \nc{\calW}{{\mathcal W}}
 \nc{\calX}{{\mathcal X}}
 \nc{\calY}{{\mathcal Y}}
 \nc{\calZ}{{\mathcal Z}}
  \nc{\cala}{{\mathcal a}}
 \nc{\calb}{{\mathcal b}}
 \nc{\calc}{{\mathcal c}}
 \nc{\cald}{{\mathcal d}}
 \nc{\cale}{{\mathcal e}}
 \nc{\calf}{{\mathcal f}}
 \nc{\calg}{{\mathcal g}}
 \nc{\calh}{{\mathcal h}}
 \nc{\cali}{{\mathcal i}}
 \nc{\calj}{{\mathcal j}}
 \nc{\calk}{{\mathcal k}}
 \nc{\call}{{\mathcal l}}
 \nc{\calm}{{\mathcal m}}
 \nc{\caln}{{\mathcal n}}
 \nc{\calo}{{\mathcal o}}
 \nc{\calp}{{\mathsf p}}
 \nc{\calq}{{\mathcal q}}
 \nc{\calr}{{\mathcal r}}
 \nc{\cals}{{\mathcal s}}
 \nc{\calt}{{\mathcal t}}
 \nc{\calu}{{\mathcal u}}
 \nc{\calv}{{\mathcal v}}
 \nc{\calw}{{\mathcal w}}
 \nc{\calx}{{\mathcal x}}
 \nc{\caly}{{\mathcal y}}
 \nc{\calz}{{\mathcal z}}

 \nc{\frakA}{{\mathfrak A}}
 \nc{\frakB}{{\mathfrak B}}
 \nc{\frakC}{{\mathfrak C}}
 \nc{\frakD}{{\mathfrak D}}
 \nc{\frakE}{{\mathfrak E}}
 \nc{\frakF}{{\mathfrak F}}
 \nc{\frakG}{{\mathfrak G}}
 \nc{\frakH}{{\mathfrak H}}
 \nc{\frakI}{{\mathfrak I}}
 \nc{\frakJ}{{\mathfrak J}}
 \nc{\frakK}{{\mathfrak K}}
 \nc{\frakL}{{\mathfrak L}}
 \nc{\frakM}{{\mathfrak M}}
 \nc{\frakN}{{\mathfrak N}}
 \nc{\frakO}{{\mathfrak O}}
 \nc{\frakP}{{\mathfrak P}}
 \nc{\frakQ}{{\mathfrak Q}}
 \nc{\frakR}{{\mathfrak R}}
 \nc{\frakS}{{\mathfrak S}}
 \nc{\frakT}{{\mathfrak T}}
 \nc{\frakU}{{\mathfrak U}}
 \nc{\frakV}{{\mathfrak V}}
 \nc{\frakW}{{\mathfrak W}}
 \nc{\frakX}{{\mathfrak X}}
 \nc{\frakY}{{\mathfrak Y}}
 \nc{\frakZ}{{\mathfrak Z}}
 \nc{\fraka}{{\mathfrak a}}
 \nc{\frakb}{{\mathfrak b}}
 \nc{\frakc}{{\mathfrak c}}
 \nc{\frakd}{{\mathfrak d}}
 \nc{\frake}{{\mathfrak e}}
 \nc{\frakf}{{\mathfrak f}}
 \nc{\frakg}{{\mathfrak g}}
 \nc{\frakh}{{\mathfrak h}}
 \nc{\fraki}{{\mathfrak i}}
 \nc{\frakj}{{\mathfrak j}}
 \nc{\frakk}{{\mathfrak k}}
 \nc{\frakl}{{\mathfrak l}}
 \nc{\frakm}{{\mathfrak m}}
 \nc{\frakn}{{\mathfrak n}}
 \nc{\frako}{{\mathfrak o}}
 \nc{\frakp}{{\mathfrak p}}
 \nc{\frakq}{{\mathfrak q}}
 \nc{\frakr}{{\mathfrak r}}
 \nc{\fraks}{{\mathfrak s}}
 \nc{\frakt}{{\mathfrak t}}
 \nc{\fraku}{{\mathfrak u}}
 \nc{\frakv}{{\mathfrak v}}
 \nc{\frakw}{{\mathfrak w}}
 \nc{\frakx}{{\mathfrak x}}
 \nc{\fraky}{{\mathfrak y}}
 \nc{\frakz}{{\mathfrak z}}

 \nc{\sha}{{\mbox{\cyr x}}}
\nc{\slfour}{{{\mathfrak{sl}}(4)}}
\nc{\slthree}{{{\mathfrak{sl}}(3)}}
\nc{\sltwo}{{{\mathfrak{sl}}(2)}}
 \nc{\sld}{{{\mathfrak{sl}}(d+1)}}
 \nc{\slr}{{{\mathfrak{sl}}(r+1)}}
 \nc{\slrr}{{{\mathfrak{sl}}(r+2)}}
\nc{\so}{{{\mathfrak {so}}(5)}}
\nc{\sofour}{{{\mathfrak {so}}(4)}}
\nc{\sothree}{{{\mathfrak {so}}(3)}}
\nc{\sosix}{{{\mathfrak {so}}(6)}}
\nc{\sotwo}{{{\mathfrak {so}}(2)}}
\nc{\spone}{{{\mathfrak {sp}}(1)}}
\nc{\sptwo}{{{\mathfrak {sp}}(2)}}
 \nc{\gtwo}{{\frg_2}}
\newcommand{\frg}{{\mathfrak g}}
 \nc{\MPV}{{\mathcal{MPV}}}

 \nc{\uds}{{\underline{s}}}
\nc{\va}{{\vec a}}
\nc{\vb}{{\vec b}}
\nc{\vc}{{\vec c}}
\nc{\vdta}{{\vec \delta}}
\nc{\ve}{{\vec e}}
\nc{\vm}{{\vec m}}
\nc{\vp}{{\vec p}}
\nc{\vn}{{\vec n}}
\nc{\vmu}{{\vec \mu}}
\nc{\vr}{{\vec r}}
\nc{\vs}{{\vec s}}
\nc{\vt}{{\vec t}}
\nc{\vu}{{\vec u}}
\nc{\vx}{{\vec x}}
\nc{\vC}{{\vec C}}
\nc{\vv}{{\bf v}}

\begin{document}

\title[Witten multiple zeta function of $G_2$ type]
{Multi-polylogs at twelfth roots of unity and\\
special values of Witten multiple zeta function \\
attached to the exceptional Lie algebra $\frg_2$}



\maketitle
\begin{center}
\sc{Jianqiang Zhao}\\
\  \\
Department of Mathematics, Eckerd College, St. Petersburg, FL 33711, USA\\
Max-Planck Institut f\"ur Mathematik, Vivatsgasse 7, 53111 Bonn, Germany
\end{center}

\bigskip
\noindent{\small {\bf Abstract.}
In this note we shall study the Witten multiple zeta function
associated to the exceptional Lie algebra $\frg_2$. Our
main result shows that its special values at nonnegative integers
can always be expressed as rational linear combinations of
the multi-polylogs evaluated at 12th roots of unity, except for
two irregular cases.}

\vskip0.6cm

\section{Introduction and preliminaries}

In \cite{W} Witten related the volumes of the moduli spaces of
representations of the fundamental groups of two dimensional surfaces
to the special values of the following zeta function attached to
complex semisimple Lie algebras $\frg$ at positive integers:
$$\gz_W(s;\frg)=\sum_\gf \frac1{(\dim \gf)^s},$$
where $\gf$ runs over all finite dimensional irreducible representations of $\frg$.
Matsumoto and his collaborators recently defined the multiple variable
analogs of $\gz_W(s;\frg)$ and studied some of their analytical and
arithmetical properties (see \cite{KMT2,KMT3,Mats2}). Let
$\gD_+$ be the set of all positive roots of $\frg$ and
$\{\gl_1,\dots,\gl_d\}$ its fundamental weights.
Then one can define (see \cite[(1.6)]{KMT2}) the multiple
zeta function attached to $\frg$ by setting
\begin{equation}\label{equ:genDef}
\gz_\frg(\{s_\ga\}_{\ga\in \gD_+}):= \sum_{m_1,\dots,m_d=1}^\infty
 \prod_{\ga\in \gD_+}\langle\ga^\vee, m_1\gl_1+\dots+m_d\gl_d\rangle^{-s_\ga},
\end{equation}
where $\ga^\vee$ is the coroot of $\ga$. The relation to the
original (single variable) Witten zeta function $\gz_W(s;\frg)$ is
given by the identity ((1.4) of loc.\ cit.)
$$\gz_W(s;\frg)=\Big(\prod_{\ga\in \gD_+} \langle \ga^\vee, \gl_1+\dots+\gl_d\rangle\Big)^s \cdot \gz_\frg(s,\dots,s).$$

In \cite{Zso5} and \cite{Zgenzeta}
we considered the special values of the Witten multiple zeta function
attached to $\so$ and $\slfour$, respectively. The main results state that
the special values of $\gz_\slfour$ (resp.~$\gz_\so$) at nonnegative
integers can be expressed as $\Q$-linear combinations of multiple zeta
values~\eqref{equ:MZV} (resp.~alternating Euler sums, i.e.,
multi-polylogs~\eqref{equ:mpol} evaluated at $\pm 1$). For general
Lie algebras we expect other special values of multi-polylogs at
roots of unity should appear.

In this note we turn to the exceptional Lie algebra $\frg_2$. Let $\ga_1$ and
$\ga_2$ be the two fundamental roots of $\frg_2$. Then the set of the positive
roots is (see \cite[p.\ 220]{Bour})
$$\gD_+=\{\ga_1,\ga_2,\ga_1+\ga_2,\ga_1+2\ga_2,\ga_1+3\ga_2,2\ga_1+3\ga_2\}.$$
Thus by definition \eqref{equ:genDef}
\begin{equation}\label{equ:gtwoDef}
 \gz_\gtwo(s_1,\dots,s_6)=\sum_{m,n=1}^\infty
 \frac{1}{m^{s_1}n^{s_2} (m+n)^{s_3} (m+2n)^{s_4}(m+3n)^{s_5}(2m+3n)^{s_6}}.
\end{equation}

To study the special values of $\gz_\gtwo$ at nonnegative integers we need
to utilize the multi-polylog function defined by
\begin{equation}\label{equ:mpol}
Li_{s_1,\dots,s_d}(x_1,\dots,x_d):=\sum_{k_1>k_2>\dots>k_d>0}
\frac{x_1^{k_1}x_2^{k_2} \dots x_d^{k_d}}{k_1^{s_1}k_2^{s_2}\cdots
k_d^{s_d}},
\end{equation}
where  $|x_1\cdots x_j|<1$ and $s_j\in \N$ for all $j$.
We call $w:=s_1+\cdots+s_d$ the \emph{weight} and $d$ the \emph{depth}.
The multi-polylog function can be
continued meromorphically to a multi-valued function on $\C^d$ (see
\cite{Zanampol}). In particular we can consider its values at roots of
unity. When $x_j=1$ for all $j$ we obtain the well-known
multiple zeta values
\begin{equation}\label{equ:MZV}
\zeta(s_1,\dots, s_d):= \sum_{k_1>\dots>k_d>0}
 \frac{1}{k_1^{s_1}\cdots k_d^{s_d}}.
\end{equation}
If all $x_j$ are $N$-th root of unity for some positive integer $N$
then we say the special value given by \eqref{equ:mpol} has \emph{level} $N$.
In this case we have
\begin{equation}\label{equ:MPOLconvCondition}
  Li_{s_1,\dots,s_d}(x_1,\dots,x_d) \text{ converges }\ \Longleftrightarrow
  \ (s_1,x_1)\ne (1,1)
\end{equation}

Let $\MPV(w,d,N)$ be the $\Q$-vector space generated by the
the multi-polylog values of weight $w$, depth $d$ and level $N$.
The main result of this note is the following
\begin{thm} \label{thm:main}
Let $s_1,\dots,s_6$ be nonnegative integers.
Let $w=s_1+\cdots+s_6$. Then $\gz_\gtwo(s_1,\dots,s_6)$
converges if and only if
\begin{equation}\label{equ:domainOfConv}
w-s_1>1,
w-s_2>1, \text{ and } \ w>2.
\end{equation}
Assume \eqref{equ:domainOfConv} holds and $(s_1,\dots,s_6)\ne
(0,0,0,0,s_5,0), (0,0,0,0,0,s_6)$. Then
$\gz_\gtwo(s_1,\dots,s_6)\in \MPV(w,d\le 2,12)$,
except for the two cases
$(s_1,\dots,s_6)=(0,0,s_3,0,0,0), (0,0,0,s_4,0,0)$
when $\gz_\gtwo(s_1,\dots,s_6)\in \langle\gz(w-1),\gz(w)\rangle_\Q$.
\end{thm}

We are going to need the following combinatorial lemma
to prove Theorem \ref{thm:main}.
\begin{lem} \label{lem:combLem}
Let $s,t$ be two positive integers. Let $x$ and $y$
be two non-zero real numbers such that $x+y\ne 0$. Then
$$ \frac{1}{x^sy^t} =
\sum_{a=0}^{s-1}  {t+a-1\choose a}\frac{1}{x^{s-a}(x+y)^{t+a}}
+\sum_{b=0}^{t-1} {s+b-1\choose b}\frac{1}{y^{t-b}(x+y)^{s+b}}.$$
\end{lem}
\begin{proof}
Take $r=2$ in \cite[Lemma 1]{ZB}.
\end{proof}

\section{Proof of Theorem \ref{thm:main}}
Before starting the proof we want to point out why it is necessary
to exclude the two cases $(s_1,\dots,s_6)=
(0,0,0,0,s_5,0), (0,0,0,0,0,s_6)$. Clearly
\begin{align*}
\gz_\gtwo(0,0,0,0,s,0)=&\sum_{m,n}^\infty \left(
    \frac{1}{(3(m+n))^s}+\frac{1}{(3(m+n)-1)^s} +\frac{1}{(3(m+n)-2)^s}\right)\\
=&\sum_{k=1}^\infty \left(
    \frac{k-1}{(3k)^s}+\frac{k-1}{(3k-1)^s} +\frac{k-1}{(3k-2)^s}\right)\\
=&\frac13\left(\gz(s-1)-\sum_{k=1}^\infty
    \frac{3}{(3k)^s}+\frac{2}{(3k-1)^s} +\frac{1}{(3k-2)^s}\right) \\
=&\frac13\left(\gz(s-1)-2\gz(s)-\sum_{k=1}^\infty
    \frac{1}{(3k)^s}+\sum_{k=1}^\infty \frac{1}{(3k-2)^s}\right) \\
=&\frac13\gz(s-1)-\frac23\gz(s)-\frac1{3^{s+1}}\gz(s)-\frac13\ga(s)
\end{align*}
where  $\ga(s)=\sum_{k=1}^\infty1/(3k-2)^s$. Fix a third root of
unity $\nu=\exp(2\pi i/3)$. Then
\begin{equation*}
 \ga(s)=  \frac13\sum_{k=1}^\infty \frac{1+\nu^{k+2}+\nu^{2k+4}}{k^s}
 = \frac13 \Big(\gz(s)+\nu^2 Li_s(\nu)+\nu Li_s(\nu^2) \Big)
\end{equation*}
By numerical evidence we believe
$\nu^2 Li_s(\nu)+\nu Li_s(\nu^2)\notin \MPV(s,d\le 2,12)$
for all positive integer $s$. For $s=1$ we have
$$\nu^2 Li_1(\nu)+\nu Li_1(\nu^2)=\frac{\sqrt{3}\pi}6+\frac{\log 3}4.$$
Now the real values in  $\MPV(1,d\le 2,12)$ are generated by
$\log(2\cos (k\pi/12))$, $k=0,\dots,5$, i.e., generated by
$$\log 2,\ \log 3,\text{ and }\ \log (\sqrt{3}+1),$$
since $\cos(\pi/12)=(\sqrt{6}+\sqrt{2})/4$ and $\cos(5\pi/12)=(\sqrt{6}-\sqrt{2})/4=\cos(\pi/12)/4.$
Similar analysis can be applied to $\gz_\gtwo(0,0,0,0,0,s_6)$.

\medskip
Let's turn to the proof of Theorem \ref{thm:main}.
The domain of convergence \eqref{equ:domainOfConv}
follows from the proof of \cite[Prop.\ 2.1]{Zgenzeta}.
We now use step-by-step reductions to prove the
rest of the theorem. First notice that if
$(s_1,\dots,s_6)=(0,0,s_3,0,0,0), (0,0,0,s_4,0,0)$
then $\gz_\gtwo(s_1,\dots,s_6)=\gz_\so(s_1,\dots,s_4)$
and the claim of the theorem is proved in \cite[Case (i) and (iii.c)]{Zso5}.
So in what follows we may assume that $(s_1,\dots,s_6)$ is \emph{regular}, i.e.,
$(s_1,\dots,s_6)$ satisfies \eqref{equ:domainOfConv} and
$$(s_1,\dots,s_6)\ne (0,0,s_3,0,0,0), (0,0,0,s_4,0,0), (0,0,0,0,s_5,0),
(0,0,0,0,0,s_6).$$

\subsection{Reduction to eight cases of three nonzero variables by Lemma \ref{lem:combLem}}
\subsubsection{General case to five nonzero variables}
Taking $x=m$ and $y=n$ in Lemma \ref{lem:combLem} we see that
\begin{align}
 \gz_\gtwo(s_1,\dots,s_6)=&
\sum_{a_1=0}^{s_1-1} {s_2+a_1-1\choose a_1}
\gz_\gtwo(s_1-a_1,0,s_3+s_2+a_1,s_4,s_5,s_6) \label{equ:1}\\
+ &\sum_{a_2=0}^{s_2-1} {s_1+a_2-1\choose a_2}
\gz_\gtwo(0,s_2-a_2,s_3+s_1+a_2,s_4,s_5,s_6). \label{equ:2}
\end{align}
\subsubsection{Five nonzero variables to four nonzero variables}
With $(x,y)=(m+3n,2m+3n)$ in Lemma \ref{lem:combLem}
(since $x+y=3(m+2n)$) we may reduce \eqref{equ:1} to
following two types of special values
\begin{align} \label{equ:11}
 \gz_\gtwo(s_1,0,s_3,s_4,s_5,0)=&\sum_{m,n=1}^\infty
 \frac{1}{m^{s_1} (m+n)^{s_3} (m+2n)^{s_4}(m+3n)^{s_5} },\\
 \label{equ:12}
\gz_\gtwo(s_1,0,s_3,s_4,0,s_6)=&\sum_{m,n=1}^\infty
 \frac{1}{m^{s_1} (m+n)^{s_3} (m+2n)^{s_4}(2m+3n)^{s_6}},
\end{align}
and reduce  \eqref{equ:2} to
\begin{align} \label{equ:21}
 \gz_\gtwo(0,s_2,s_3,s_4,s_5,0)=&\sum_{m,n=1}^\infty
 \frac{1}{n^{s_2} (m+n)^{s_3} (m+2n)^{s_4}(m+3n)^{s_5} },\\
 \label{equ:22}
\gz_\gtwo(0,s_2,s_3,s_4,0,s_6)=&\sum_{m,n=1}^\infty
 \frac{1}{n^{s_2} (m+n)^{s_3} (m+2n)^{s_4}(2m+3n)^{s_6}},
\end{align}
Indeed, we have
\begin{align*}
 \gz_\gtwo(s_1,0,s_3,\dots,s_6)=&
\sum_{a_5=0}^{s_5-1} {s_6+a_5-1\choose a_5}\frac1{3^{s_6+a_5}}
\gz_\gtwo(s_1,0,s_3,s_4+s_6+a_5,s_5-a_5,0)  \\
+ &\sum_{a_6=0}^{s_6-1} {s_5+a_6-1\choose a_6}\frac1{3^{s_5+a_6}}
\gz_\gtwo(s_1,0,s_3,s_4+s_5+a_6,0,s_6-a_6)
\end{align*}
and similarly
\begin{align*}
 \gz_\gtwo(0,s_2,s_3,\dots,s_6)=&
\sum_{a_5=0}^{s_5-1} {s_6+a_5-1\choose a_5}\frac1{3^{s_6+a_5}}
\gz_\gtwo(0,s_2,s_3,s_4+s_6+a_5,s_5-a_5,0)  \\
+ &  \sum_{a_6=0}^{s_6-1} {s_5+a_6-1\choose a_6}\frac1{3^{s_5+a_6}}
\gz_\gtwo(0,s_2,s_3,s_4+s_5+a_6,0,s_6-a_6).
\end{align*}

\subsubsection{Four nonzero variables to three nonzero variables}
For any positive integers $a$, $b$, $s_1,s_2,s_3$ of which at least
two of $s_j$'s are nonzero such that $s_2+s_3>1$
and $s_1+s_2+s_3>2$ we define
\begin{equation}\label{equ:A}
 A_{a,b}(s_1,s_2,s_3):=\sum_{m,n=1}^\infty
 \frac{1}{m^{s_1} ((a+1)m+bn)^{s_2} (am+bn)^{s_3}}.
\end{equation}
We provide the details for reduction of \eqref{equ:11} and leave
that of \eqref{equ:12}--\eqref{equ:22} to the interested reader.
Applying Lemma \ref{lem:combLem} with $(x,y)=(m+n,m+3n)$ to \eqref{equ:11}
we can reduce it as follows:
\begin{align}
 \gz_\gtwo(s_1,0,s_3,s_4,s_5,0)=&
\sum_{a_3=0}^{s_3-1}{s_5+a_3-1\choose a_3} \frac1{2^{s_5+a_3}}
\gz_\gtwo(s_1,0,s_3-a_3,s_4+s_5+a_3,0,0)\notag \\
+ &\sum_{a_5=0}^{s_5-1}{s_3+a_5-1\choose a_5}\frac1{2^{s_3+a_5}}
\gz_\gtwo(s_1,0,0,s_4+s_3+a_5,s_5-a_5,0) \notag \\
=&\sum_{a_3=0}^{s_3-1} {s_5+a_3-1\choose a_3}
  2^{s_3-s_5-2a_3}A_{1,2}(s_1,s_3-a_3,s_4+s_5+a_3)
 \label{equ:111} \\
+ & \sum_{a_5=0}^{s_5-1} {s_3+a_5-1\choose a_5}
 3^{s_4+s_3+a_5}2^{s_5-s_3-2a_5}A_{2,6}(s_1,s_4+s_3+a_5,s_5-a_5) .\label{equ:112}
\end{align}
Applying Lemma \ref{lem:combLem} with $(x,y)=(m+n,m+3n)$ to \eqref{equ:21}
we get
\begin{align}
 \gz_\gtwo(0,s_2,s_3,s_4,s_5,0)= &
 \sum_{a_3=0}^{s_3-1} {s_5+a_3-1\choose a_3}
 \frac1{2^{s_5+a_3}} A_{1,1}(s_2,s_4+s_5+a_3,s_3-a_3)  \label{equ:211} \\
+ &\sum_{a_5=0}^{s_5-1} {s_3+a_5-1\choose a_5}
 \frac1{2^{s_3+a_5}} A_{2,1}(s_2,s_5-a_5,s_4+s_3+a_5).\label{equ:212}
\end{align}
We may apply Lemma \ref{lem:combLem} with $(x,y)=(m+n,m+2n)$
to \eqref{equ:12} to get
\begin{align}
\gz_\gtwo(s_1,0,s_3,s_4,0,s_6)=&\sum_{a_3=0}^{s_3-1} {s_4+a_3-1\choose a_3}
  3^{s_3-a_3} A_{2,3}(s_1,s_3-a_3,s_6+s_4+a_3)
 \label{equ:121} \\
+ \sum_{a_4=0}^{s_4-1} & {s_3+a_4-1\choose a_4}
 3^{s_4-a_4} 2^{s_6+s_3+a_4} A_{3,6}(s_1,s_6+s_3+a_4,s_4-a_4) .\label{equ:122}
\end{align}
and to \eqref{equ:22} to get
\begin{align}
\gz_\gtwo(0,s_2,s_3,s_4,0,s_6)=&\sum_{a_3=0}^{s_3-1} {s_4+a_3-1\choose a_3}
  2^{s_3-a_3} A_{2,2}(s_2,s_6+s_4+a_3,s_3-a_3)
 \label{equ:221} \\
+ & \sum_{a_4=0}^{s_4-1} {s_3+a_4-1\choose a_4}
 2^{s_4-a_4}  A_{3,2}(s_2,s_4-a_4,s_6+s_3+a_4) .\label{equ:222}
\end{align}

To summarize we have reduced the general special values of
$\gz_\gtwo$ to the eight cases given by \eqref{equ:111}--\eqref{equ:222}.
Moreover, if we start with a regular value then after the above reduction
steps we still get regular values.

\subsection{Computation of $A_{a,b}(s_1,s_2,s_3)$}
Throughout this subsection for any positive integer $N$
let $\mu_N$ be a fixed primitive $N$-th root of unity.
If we na\"ively apply Lemma \ref{lem:combLem} with $x=(a+1)m+bn$ and
$y=-am-bn$ to \eqref{equ:A} then we would get divergent series
because of the convergence condition
\eqref{equ:MPOLconvCondition} of the multi-polylogs.
Thus we need to modify it by calculating $A_{a,b}(s_1,s_2,s_3)$ as follows:
\begin{align*}
A_{a,b}(s_1,s_2,s_3)=&\sum_{m,n=1}^\infty
 \frac{(-1)^{s_3}}{m^{s_1}((a+1)m+bn)^{s_2} (-am-bn)^{s_3}} \notag \\
  =&\sum_{a_3=0}^{s_3-2} {s_2+a_3-1\choose a_3}(-1)^{a_3}
 B_{a,b}(s_1+s_2+a_3,s_3-a_3)\notag \\
 +&\sum_{a_2=0}^{s_2-2} {s_3+a_2-1\choose a_2}
 (-1)^{s_3}B_{a+1,b}(s_1+s_3+a_2,s_2-a_2)\notag \\
+&  \frac{(-1)^{s_3}}{b} {s_2+s_3-2\choose s_3-1}
\Big(C_{a,b}(s_1+s_2+s_3-1)+\lim_{M\to\infty} S_a^{(M)}(s_1+s_2+s_3-1)\Big)
\end{align*}
where the functions $B$, $C$ and $S$ are defined as follows.
For any positive integers  $a$, $b$, $r$ and $s$ (with $s>1$)
\begin{multline*}
B_{a,b}(r,s):= \sum_{m,n=1}^\infty \frac{1}{m^r (am+bn)^s}
 =\frac{a^{r-1}}{b}  \sum_{j=0}^{b-1}  \sum_{k=0}^{a-1} \sum_{m,n=1}^\infty
 \frac{1}{m^r} \frac{\mu_a^{km}\mu_b^{jn}}{(m+n)^s} \\
 =\frac{a^{r-1}}{b}  \sum_{j=0}^{b-1}  \sum_{k=0}^{a-1}
 Li_{s,r}(\mu_b^j, \mu_a^k \mu_b^{-j})
 \in \MPV\big(r+s,2,{\rm lcm}(a,b)\big) .
\end{multline*}
where lcm denotes the lowest common multiple.
For any positive integers $a$, $b$ and $s$ (with $s>1$)
\begin{align*}
C_{a,b}(s):=&  \sum_{j=1}^{b-1} \sum_{m,n=1}^\infty
 \frac{1}{m^s} \left(\frac{\mu_b^{jn}}{(a+1)m+n}-\frac{\mu_b^{jn}}{am+n}\right)\\
 =& \sum_{j=1}^{b-1} \left((a+1)^{s-1}\sum_{k=0}^a
 Li_{1,s}(\mu_b^j, \mu_{a+1}^k \mu_b^{-j})-a^{s-1}\sum_{k=1}^a
 Li_{1,s}(\mu_b^j, \mu_a^k \mu_b^{-j})\right)\\
 \in & \MPV\big(s+1,2,{\rm lcm}(a,a+1,b)\big).
\end{align*}
For any positive integers $a$ and $s$ (with $s>1$)
\begin{align*}
 S_a^{(M)}(s):=&\sum_{m,n=1}^M \frac{1}{m^s} \left(\frac1{(a+1)m+n}-\frac1{am+n}\right)\\
=& \sum_{m,n=1}^M \frac{1}{m^s} \left(\frac{1}{(a+1)m+n}-\frac{1}{am+n}\right)\\
  =&\sum_{m=1}^M
\frac{1}{m^s} \left(  \sum_{n=1+(a+1)m}^{M+(a+1)m}-\sum_{n=1+am}^{M+am}\right)
     \frac{1}{n} \\
=&\sum_{m=1}^{M}\frac{1}{m^s} \left( \sum_{n=M+am+1}^{M+(a+1)m}
+\sum_{n=1}^{am}-\sum_{n=1}^{(a+1)m} \right) \frac{1}{n} .
\end{align*}
Noticing that $s>1$ and therefore
\begin{equation*}
 \sum_{m=1}^M\frac{1}{m^s} \sum_{n=M+am+1}^{M+(a+1)m} \frac{1}{n}
<\sum_{m=1}^M \frac{1}{m^{s-1} M}\ll \frac{\log M}{M}\to 0
\quad\text{ as } \quad M\to \infty,
\end{equation*}
we quickly see that
\begin{multline*}
S_a(s):=\lim_{M\to\infty}S_a^{(M)}(s)
 =a^{s-1}\sum_{k=1}^a \sum_{m\ge n\ge 1}\frac{\mu_a^{km}}{m^s n}
 -(a+1)^{s-1}\sum_{k=1}^{a+1} \sum_{ m\ge n\ge 1}
\frac{\mu_{a+1}^{km}}{m^s n}\\
=a^{s-1}\sum_{k=1}^a \Big(Li_{s,1}(\mu_a^k,1)
+Li_{s+1}(\mu_a^k)\Big)
-(a+1)^{s-1}\sum_{k=1}^{a+1}  \Big(Li_{s,1}(\mu_{a+1}^k,1)+Li_{s+1}(\mu_{a+1}^k)\Big).
\end{multline*}
By the expressions for $B_{a,b}(r,s)$,
$C_{a,b}(s)$ and $S_a(s)$ we clearly have
$$ A_{a,b}(s_1,s_2,s_3)\in \MPV\big(s_1+s_2+s_3,d\le 2,{\rm lcm}(a,a+1,b)\big).$$
Therefore Theorem \ref{thm:main} follows from
\eqref{equ:111}--\eqref{equ:222} immediately.
This concludes the proof of Theorem~\ref{thm:main}.

\section{Some Numerical Examples}
We remark that the twelfth root of unity $\mu_{12}$
appears essentially in the computation of
\eqref{equ:122}. For example we have
\begin{align*}
\gz_\gtwo(1,0,0,1,0,1)=&\sum_{m,n=1}^\infty\frac{1}{m(m+2n)(2m+3n)}
     =-C_{3,6}(2)-S_3(2)\Big)\\
     =&\sum_{j=1}^5 \left(3\sum_{k=1}^3
 Li_{1,2}(\mu_6^j, \mu_3^k \mu_6^{-j})-4\sum_{k=0}^3
 Li_{1,2}(\mu_6^j, \mu_{4}^k \mu_6^{-j})\right)\\
+&4\sum_{k=1}^{4}  \Big(Li_{2,1}(\mu_{4}^k,1)+Li_3(\mu_{4}^k)\Big)
-3\sum_{k=1}^3 \Big(Li_{2,1}(\mu_3^k,1)+Li_3(\mu_3^k)\Big)\\
=&
0.29118204074051670279981404910049215137872186455887\backslash\\
&\phantom{0.}44495517845170396367225951404279842529727941799295
\end{align*}
with the error bounded by $10^{-100}$ using the computer program
GiNac developed by Vollinga and Weinzierl \cite{GiNac}.
With our current state of knowledge it is very difficult to
prove rigorously that this value does not lie in
$\MPV\big(3,d\le 2,N\big)$ for any proper divisor $N$ of 12.
However, by assuming Grothendieck's period conjecture it might be
possible to prove it by finding explicitly the bases
of $\MPV\big(3,d\le 2,N\big)$ using the ideas of \cite{Zocta,Zpolrel}.
To illustrate this let us fix a third root of
unity $\nu=\exp(2\pi i/3)$ as before and consider
\begin{align*}
 \gz_\gtwo(1,0,0,0,2,0)=\sum_{m,n=1}^\infty\frac{1}{m(m+3n)^2}
     =&\frac13\sum_{m,n=1}^\infty\frac{1+\nu^{n}+\nu^{2n}}{m(m+n)^2}\\
     =&\frac13\Big(\gz(2,1)+Li_{2,1}(\nu,\nu^2)+Li_{2,1}(\nu^2,\nu)\Big)
\end{align*}
which lies in the $\Q$-vector space $\MPV(3,3)$ generated by the special
values of multi-polylog of weight three and level three.
Assuming a variant of Grothendieck's period conjecture, Deligne \cite{Del}
constructs explicitly a set of basis for $\MPV(w,N)$ and in particular
predicts that $\dim_\Q\MPV(3,3)=8$.  In \cite{Zpolrel}
by using (regularized) double shuffle relations (RDS) we can explicitly
construct the following basis of $\MPV(3,3)$:
\begin{align*}
&Li_{2,1}(\nu,1),\quad Li_{1,1,1}(\nu,\nu,1),
\quad  Li_{1,1,1}(\nu^2,1,\nu^2),\quad
Li_{1,1,1}(\nu^2,\nu^2,\nu^2),\\
&Li_{2,1}(1,\nu),\quad  Li_{1,1,1}(\nu,\nu,\nu),\quad
Li_{1,1,1}(\nu^2,\nu^2,1),\quad Li_{1,1,1}(\nu,\nu^2,\nu^2).
\end{align*}
Further we can show by using RDS that
\begin{align*}
\gz(3)=&63L_{1,1,1}(\nu,\nu,1)+63L_{1,1,1}(\nu^2,\nu^2,\nu^2)
+18L_{1,1,1}(\nu^2,1,\nu^2)-126L_{2,1}(\nu,1)\\
-&81L_{1,1,1}(\nu^2,\nu^2,1)-9L_{1,1,1}(\nu,\nu,\nu)
-54L_{1,1,1}(\nu,\nu^2,\nu^2)-18L_{2,1}(1,\nu),
\end{align*}
while
\begin{align*}
\gz_\gtwo(1,0,0,0,2,0)=
& 52L_{1,1,1}(\nu^2,\nu^2,1)+84L_{2,1}(\nu,1)
-\frac{112}{3} L_{1,1,1}(\nu^2,\nu^2,\nu^2)
  +\frac{8}{3} L_{1,1,1}(\nu,\nu,\nu)\\
+&\frac{116}{3} L_{1,1,1}(\nu,\nu^2,\nu^2)+12L_{2,1}(1,\nu)
-\frac{124}{3}L_{1,1,1}(\nu,\nu,1)
 -\frac{44}{3} L_{1,1,1}(\nu^2,1,\nu^2).
\end{align*}
Clearly $\gz_\gtwo(1,0,0,0,2,0)$ is not a rational multiple of $\gz(3)$
(under the assumption of a variant of Grothendieck's period conjecture)
since otherwise $\dim_\Q\MPV(3,3)<8$.
Therefore $\gz_\gtwo(1,0,0,0,1,0)$ cannot lie in $\MPV(3,1)$
which is generated by $\gz(3)=\gz(2,1)$.

By applying our reduction process in the proof of Theorem~\ref{thm:main}
and then using GiNac we may also compute numerically
\begin{align*}
\gz_\gtwo(1,1,1,1,1,1)=&
0.01110008020419022277039637092289017833699879038194\backslash\\
&\phantom{0.}45086731059450527930767014636924683152705726916209\dots
\end{align*}
and verify directly that
$$\sum_{m,n=1}^{300} \frac{1}{mn(m+n)(m+2n)(m+3n)(2m+3n)}
=0.01110008013\dots.$$
Similarly
\begin{align*}
\gz_\gtwo(2,1,1,1,1,1)=&
0.00995272345287837349624059820619791150631359962925\backslash\\
&\phantom{0.}76058197642267125491591895578259027698791211370695\dots\\ 
\gz_\gtwo(1,2,1,1,1,1)=&
0.01051743558635248267821710084904837131835692774970\backslash\\
&\phantom{0.}58926903804614080083290282688685278848741106023759\dots\\
\gz_\gtwo(1,1,2,1,1,1)=&
0.00497203096318456908722247199264938834521998397322\backslash\\
&\phantom{0.}59104469026433274260043289411830973101381601191180\dots\\
\gz_\gtwo(1,1,1,2,1,1)=&
0.00334830993415208689505176677618470935993308677601\backslash\\
&\phantom{0.}07755998804241800305039265855252311030101487133354\dots\\
\gz_\gtwo(1,1,1,1,2,1)=&
0.00252989581107351464804145549755978591413008088242\backslash\\
&\phantom{0.}72342790902902770229510932678198294307546061097483\dots\\
\gz_\gtwo(1,1,1,1,1,2)=&
0.00199953266044219311834226765773032132006109764665\backslash\\
&\phantom{0.}46802961304161882193488113910506120491878470601712\dots\\
\end{align*}
Computation by Maple of the partial sums $\sum_{m,n=1}^{300}$
of the above values
shows that they agree with each other for at least the first ten
decimal digits. Finally, the most interesting special value is
\begin{align*}
\gz_\gtwo(2,2,2,2,2,2)=&
0.00007135906438752907355938750633729102655371125360\backslash\\
&\phantom{0.}34722293695865502120914689101295543182083895277791\dots.
\end{align*}
By the linear relation detection program PSLQ implemented
with EZ-face \cite{EZface} we find that
\begin{equation}\label{equ:222222}
    \gz_\gtwo(2,2,2,2,2,2)\overset{?}{=}
    \frac{2^6\cdot 5^3\cdot 23}{3^6\cdot 7^2\cdot11\cdot13}\gz(2)^6
    =\frac{5^3\cdot 23}{3^{12}\cdot 7^2\cdot11\cdot13}\pi^{12}
\end{equation}
with the error bounded by $10^{-100}$. This is in agreement with the
original discovery of Witten that for any positive integer $m$
$\gz_\gtwo(\{2m\}_6)=c_m\pi^{12m}$ for some $c_m\in\Q$.
It might be possible to use the techniques developed in \cite{KMT2,KMT3}
to find a closed formula for $c_m$ involving Bernoulli numbers.

\bigskip
\noindent
\textbf{Acknowledgement.}
The author would like to thank the Max-Planck-Institut
f\"ur Mathematik for providing financial support
during his sabbatical leave when this work was done.
He also greatly benefited from several discussions with Andrey Levin,
Don Zagier and Wadim Zudilin, especially Levin's suggestion
to extend the results in \cite{Zso5,Zgenzeta}
to the exceptional Lie algebra $\frg_2$.


\begin{thebibliography}{99}
\bibitem{EZface}
J.~Borwein, P.~Lisonek, and P.~Irvine,\emph{ An interface for
evaluation of Euler sums}. Available online at
\url{http://oldweb.cecm.sfu.ca/cgi-bin/EZFace/zetaform.cgi}

\bibitem{Bour}
N.\ Bourbaki, \emph{Groupe et Alg\`ebres de Lie, Chapitres 4, 5 et 6},
Masson, Paris, 1981.

\bibitem{Del}
P.~Deligne, \emph{Le groupe fondamental de la $\G_m-\mmu_N$},
unpublished manuscript.

\bibitem{KMT2}
Y.\ Komori, K.\ Matsumoto and H.\ Tsumura,
\emph{On Witten multiple zeta-functions associated with semisimple Lie algebras II}.
Preprint.

\bibitem{KMT3}
Y.\ Komori, K.\ Matsumoto and H.\ Tsumura,
\emph{On Witten multiple zeta-functions associated with semisimple Lie algebras III}.
Preprint.

\bibitem{Mats2}
K.~Matsumoto and H.\ Tsumura,
\emph{On Witten multiple zeta-functions associated with semisimple Lie algebras I},
Ann.\ Inst.\ Fourier \textbf{56} (2006), 1457--1504.

\bibitem{GiNac}
J.\ Vollinga, S.\ Weinzierl, \emph{Numerical evaluation of multiple polylogarithms},
arXiv:hep-ph/0410259

\bibitem{W}
E.\ Witten, \emph{On quantum gauge theories in two-dimensions.}
Commun.\ Math.\ Phys.\ \textbf{141}(1) (1991), 153--209.

\bibitem{Zag} D.~Zagier, \emph{Values of zeta
function and their applications}. Proc.\ of the First
European Congress of Math. \textbf{2} (1994), 497--512.

\bibitem{Zocta}
J.~Zhao, \emph{Multiple polylogarithm values
at roots of unity}. C.\ R.\ Acad.\ Sci.\ Paris, Ser. I.
\textbf{346} (2008), 1029--1032.

\bibitem{Zpolrel}
J.~Zhao, \emph{Standard relations of multiple polylogarithm values at
roots of unity}. Submitted.

\bibitem{Zanampol}
J.~Zhao, \emph{Analytic continuation of multiple polylogarithms}. Analysis
Mathematica, \textbf{33} (2007), 301--323.

\bibitem{Zso5}
J.\ Zhao, \emph{Alternating Euler sums and special values of Witten
multiple zeta function attached to $\so$}, arxiv: 0903.0473

\bibitem{Zgenzeta}
J.\ Zhao and X.\ Zhou, \emph{Witten multiple zeta values attached to $\slfour$}. arxiv: 0903.2383

\bibitem{ZB}
X.~Zhou and D.M.~Bradley,
\emph{On Mordell-Tornheim sums and multiple zeta values}. Preprint, 2008.
\end{thebibliography}
\end{document}